\documentclass{cimart}

\usepackage{enumitem}

\numberwithin{equation}{section}
\DeclareMathOperator{\ms}{meas}
\newcommand{\R}{\mathbb{R}}
\newcommand{\RN}{\R^N}
\newcommand{\norm}[1]{\|#1\|}
\newcommand{\essinf}{\mathop{\mathrm{ess\,inf}}}
\newcommand{\esssup}{\mathop{\mathrm{ess\,sup}}}
\newcommand{\Lp}[1]{L^{#1}(\RN)}
\newcommand{\Linf}{L^{\infty}(\RN)}
\newcommand{\Wp}[1]{W^{1,#1}(\RN)}
\newcommand{\diff}{\,\mathrm{d}}
\newcommand{\eps}{\varepsilon}
\newcommand{\intr}{\int_{\RN}}

\title{%
    Existence of three solutions for parametric $p$-Laplacian equations in $\mathbb{R}^N$
    }

\authors{%
    Eleonora Amoroso and Giuseppina D'Aguì
    }

\authorinfo[%
    E. Amoroso]{%
    Department of Engineering, University of Messina, Italy}{%
    eleonora.amoroso@unime.it
    }

\authorinfo[%
    G. D'Aguì]{%
    Department of Engineering, University of Messina, Italy}{%
    dagui@unime.it
    }

\abstract{%
    In this paper we deal with a nonlinear elliptic problem in the whole space $\mathbb{R}^N$ involving the $p$-Laplacian operator with $p>N$. Under suitable assumptions on the potential term and on the nonlinearity, we establish the existence of at least three distinct weak solutions. Our approach is variational and relies on two different three critical points theorems. The results extend to higher dimensions some recent multiplicity theorems for one-dimensional $p$-Laplacian equations on the real line, and provide explicit ranges for the parameter $\lambda > 0$ ensuring the existence of multiple solutions, which are nonnegative under additional sign conditions. The nontriviality of the solutions is also discussed. Finally, applications to problems whose nonlinearities have separated variables are presented.
    }



\acknowledgments{%
    The authors are members of the Gruppo Nazionale per l'Analisi Matematica, la Probabilit\`a e le loro Applicazioni (GNAMPA) of the Istituto Nazionale di Alta Matematica (INdAM).
    }

\VOLUME{35}
\YEAR{2027}
\ISSUE{1}
\NUMBER{3}
\DOI{https://doi.org/10.46298/cm.19536}
\licence{CC BY-SA 4.0}
\editinfo{September 30, 2026}{November 4, 2026}{Serena Dipierro}

\begin{document}

\section{Introduction} 
	Equations involving the $p$-Laplacian operator $\Delta_p u = \mathrm{div} (|\nabla u|^{p-2}\nabla u)$ play a central role in nonlinear analysis and have been widely investigated in recent years. Among the various settings, the study of such equations in the whole space $\RN$ is of particular interest, since it raises specific analytical challenges mainly due to the lack of compactness of the embedding of Sobolev spaces into suitable function spaces.
	 It is worth noting that the case $p < N$ in $\mathbb{R}^N$ has been extensively studied in the literature, see for instance \cite{Ambrosetti-Garcia-Peral, Barletta-Candito-Marano-Perera, Drabek-Huang, Gu-Zeng-Zhou, Wang-Sun, Zhang-Zhang} and the references therein. By contrast, the case $p > N$ is much less explored. The paper by Amoroso--Bonanno--Perera \cite{Amoroso-Bonanno-Perera-2023} represents, to the best of our knowledge, the first systematic study of problem \eqref{or_prob} below in the case $p > N$, obtaining results comparable to those in the bounded domain case and overcoming the issues that arise in the unbounded setting.
	
	This paper represents a further development of the study conducted in \cite{Amoroso-Bonanno-Perera-2023, Amoroso-Bonanno-Perera-2025}, where results on the existence of one, two, and infinitely many solutions were established for the following nonlinear elliptic problem
		\begin{equation} \label{or_prob}
		\begin{cases}
			-\Delta_p u + a(x) |u|^{p-2} u = \lambda f(x,u) \quad \text{in} \;  \RN,\\
			\lim\limits_{|x| \to + \infty} u(x) =0,
		\end{cases}
	\end{equation}
	where $p > N$, $\lambda>0$ is a parameter, $a \colon \RN \to \R$ is a bounded potential with positive essential infimum and $f\colon\R^N \times \R \to \R$ is an $L^1$-Carathéodory function (see assumption \eqref{hp-problem} below for more details).

	Furthermore, the setting considered in this paper includes, as a particular case, the one studied in \cite{Bonanno-O'Regan-Vetro-2017}, where the authors established the existence of three solutions on the real line for a nonlinearity with separated variables and a constant potential, namely for the problem
	\[
	\begin{cases}
		-(|u'(x)|^{p-2} u'(x))' + B |u(x)|^{p-2} u(x) = \lambda \alpha(x) g(u(x)) \qquad \text{a.e. in}\; \R,\\
		\lim\limits_{|x| \to + \infty} u(x) =0,
	\end{cases}
	\]
	with $p > 1$,  $\lambda>0$, $B$ a positive constant, $\alpha \in L^1(\R)$ nonnegative and $g$ continuous and nonnegative. Indeed, since $N=1$ and $p>1$, this problem falls within the case $p>N$. In this sense, our results extend the framework of \cite{Bonanno-O'Regan-Vetro-2017} to higher dimensions, provided that $p>N$, and to variable potentials (see Remark~\ref{rem:BOV} for a comparison in the case $N=1$).

	Our aim is to establish sufficient conditions ensuring the existence of at least three distinct weak solutions of problem \eqref{or_prob}. The approach is variational and is based on two different three critical points theorems, due to
	Bonanno--Marano \cite{Bonanno-Marano-2010} and Averna--Bonanno \cite{Averna-Bonanno-2009}, see also Bonanno--Candito \cite{Bonanno-Candito-2008}. Three critical points theorems of this type originate from the work of Ricceri \cite{Ricceri-2000}, where the existence of an open interval of parameters for which the problem admits at least three solutions is established, without an explicit description of such an interval. The results in \cite{Averna-Bonanno-2009, Bonanno-Candito-2008, Bonanno-Marano-2010} provide instead explicit intervals of parameters, which allows us to give explicit ranges for $\lambda$ in our setting. By applying these abstract results to the energy functional associated with problem \eqref{or_prob}, we obtain the existence of three solutions for the parameter $\lambda$ lying in suitable intervals depending explicitly on the data of the problem. We emphasize that all the constants involved (in particular, the constant $k_a$ in \eqref{embedding} and the constant $K_R$ in \eqref{K_R}) are explicitly computable, so that the intervals of parameters are given in closed form.

	In particular, the first result (Theorem~\ref{3sol_1}) ensures the existence of three weak solutions under an algebraic condition on the nonlinear term and assuming a $p$-sublinear growth at infinity. As a consequence (Corollary~\ref{th1:nonneg}), if in addition $f(\cdot,0)$ is nonnegative almost everywhere in $\RN$, then problem \eqref{or_prob} admits three nonnegative weak solutions. Next, by assuming further sign and growth hypotheses, but without requiring any growth condition at infinity, we obtain the existence of three nonnegative weak solutions whose norms are bounded by an explicit constant (Theorem~\ref{3sol_2}). The question of whether one of the solutions may be the trivial one is addressed in Remark~\ref{rem:nontrivial}. As an application, we also discuss the particular case in which the nonlinearity has separated variables (Corollaries~\ref{th1-SV} and \ref{th2-SV}). In this setting, we further show (Theorem~\ref{th3-SV}) that, if the primitive of the nonlinearity is $p$-superlinear at zero and $p$-sublinear at infinity, problem \eqref{problem-SV} admits at least three distinct nonnegative weak solutions, at least two of which are nontrivial, for every $\lambda > \lambda^*$, where $\lambda^*$ is explicitly given in terms of the data. We point out that, unlike the previous results, the range of parameters here is unbounded, and the assumptions involve only the behavior of the nonlinearity at zero and at infinity.

	The paper is organized as follows. In Section~\ref{Sec:prel} we present the mathematical background and the variational setting of the problem, recalling also the two abstract three critical points theorems that will be used in the proofs. Section~\ref{Sec:main} is devoted to the statements and proofs of the main existence results. Finally, some consequences and applications to problems whose nonlinearities have separated variables are presented at the end of Section~\ref{Sec:main}.
	\section{Preliminaries} \label{Sec:prel}
Consider the Lebesgue space $\Lp{p}$ equipped with the classical norm
	\[
	\norm{u}_p = \left( \int_{\R^N} |u|^p \diff x\right)^{\frac{1}{p}}
	\]
	and the Sobolev space $\Wp{p}$, whose usual norm is
	\begin{align*}
		\norm{u}_{1,p} = \norm{u}_p + \norm{\nabla u}_p.
	\end{align*}
	The space $\Wp{p}$ is a reflexive real Banach space and, if $p>N\ge1$, Morrey's theorem \cite[Theorem 9.12]{Brezis-book-2011} ensures that $\Wp{p}$ is continuously embedded in $\Linf$, that is, there exists a constant $k>0$ such that 
	\begin{align*}
		\norm{u}_\infty \le k \norm{u}_{1,p} \qquad \text{for all } \; u \in \Wp{p}.
	\end{align*}
	To the best of our knowledge, the best explicit estimate of this constant $k$ currently available in the literature is given in Amoroso--Bonanno--Perera \cite[Proposition 2.2, Remark 2.1]{Amoroso-Bonanno-Perera-2023}. Furthermore, note that if $u \in \Wp{p}$, with $N < p < \infty$, then its continuous representative, still denoted by $u$, satisfies
	\[
	\lim\limits_{|x| \to \infty} u(x) =0,
	\]
	see Brezis \cite[Remark 12, p. 283]{Brezis-book-2011}.
	Hence, the decay condition in \eqref{or_prob} is automatically satisfied by every function in $\Wp{p}$, and problem \eqref{or_prob} can be rewritten as
	\begin{align} \label{problem} \tag{P$_\lambda$}
		\begin{cases}
		-\Delta_p u + a(x) |u|^{p-2} u = \lambda f(x,u) \quad \text{in} \quad  \RN,	\\
		u \in \Wp{p},
		\end{cases}
	\end{align}
	where we assume that
	\begin{enumerate}[label=\textnormal{(H)},ref=\textnormal{H}] 
		\item\label{hp-problem} $N< p < \infty$, $a \in \Lp{\infty}$ with $a_- = \essinf\limits_{\R^N} a >0$, and \mbox{$f\colon\R^N \times \R \to \R$} is $L^1$-Carathéodory, namely $f(\cdot, t)$ is measurable for all $t \in \R$, the function $f(x, \cdot)$ is continuous for a.e. $x \in \RN$ and $\sup\limits_{|t| \leq \sigma} |f(\cdot,t)|  \in \Lp{1}$ for all $\sigma>0$.
		\end{enumerate}
	Now, we endow the Sobolev space $\Wp{p}$ with the following equivalent norm
	\begin{align*}
		\norm{u}=\left( \displaystyle\intr |\nabla u(x)|^p \diff x + \displaystyle\int_{\R^N}a(x)|u(x)|^p \diff x \right)^{\frac{1}{p}},
	\end{align*}
	and, by Amoroso--Bonanno--Perera \cite[Lemma 2.1]{Amoroso-Bonanno-Perera-2023}, one has
	\begin{equation} \label{embedding}
		\norm{u}_\infty \le k_a \norm{u} \qquad \text{for all }\, u \in \Wp{p},
	\end{equation}
	where
	\begin{equation*}
		k_a= \left( \dfrac{1}{a_-}\right)^{\frac{p-N}{p^2}} 2^{\frac{p-1}{p}} \left(\dfrac{\Gamma\left(1+\frac{N}{2}\right)}{\pi^{\frac{N}{2}}}\right)^{\frac{1}{p}} \dfrac{1}{N^{\frac{N}{p^2}}} \left( \dfrac{p-1}{p-N}\right)^{\frac{N(p-1)}{p^2}},
	\end{equation*}
	with $\Gamma(s) = \int_0^\infty e^{-t} t^{s-1}\diff t$, $s>0$, being the well-known Gamma function.
	
	The investigation is based on variational methods and critical point theory. With this approach, the critical points of a suitable functional, called the energy functional, coincide with the weak solutions of the problem. More precisely, $u \in \Wp{p}$ is a \textit{weak solution} of \eqref{problem} if
	\[
	 \intr \left( |\nabla u|^{p-2} \nabla u \cdot \nabla v + a(x) |u|^{p-2} u v \right) \diff x = \lambda \intr f(x,u) v \diff x	\]
	for all $v \in \Wp{p}$. Set $F(x,t) = \int_0^t f(x, \xi) \diff \xi$ for every $(x,t) \in \RN \times \R$. Note that, owing to the Carath\'eodory property of $f$ (see \eqref{hp-problem}), for every $c>0$ the function $x\mapsto\max_{|t|\le c}F(x,t)$ is measurable and belongs to $\Lp{1}$. Consider the functionals \mbox{$\Phi, \Psi \colon \Wp{p} \to \R$} given by
	\[
	\Phi(u)= \dfrac{1}{p} \norm{u}^p, \qquad \Psi(u) = \intr F(x,u(x)) \diff x,
	\]
	for all $u \in \Wp{p}$. Then, the energy functional $I_\lambda$ related to problem \eqref{problem} is defined by $I_\lambda = \Phi - \lambda \Psi$ for all $\lambda>0$. The functionals $\Phi, \Psi, I_\lambda$ are continuously G\^{a}teaux differentiable and one has
	\begin{align*}
		\Phi'(u)(v) &= \intr \left( |\nabla u|^{p-2} \nabla u \cdot \nabla v + a(x) |u|^{p-2} u v \right) \diff x, \\
		\Psi'(u)(v) &= \intr f(x,u) v \diff x,
	\end{align*}
	for any $u, v \in \Wp{p}$. Therefore, $u \in \Wp{p}$ is a weak solution of problem \eqref{problem} if and only if $u$ is a critical point of $I_\lambda$, namely $I'_\lambda(u)(v)=0$ for all \mbox{$v \in \Wp{p}$}.
	
	Moreover, $\Phi$ is coercive, convex and sequentially weakly lower semicontinuous, and $\Phi'$ admits a continuous inverse on $\left(\Wp{p}\right)^*$, since $\Wp{p}$, endowed with the norm $\norm{\cdot}$, is uniformly convex. On the other hand, owing to \eqref{hp-problem}, $\Psi'$ is compact: indeed, arguing as in the proof of \cite[Lemma 2.3]{Amoroso-Bonanno-Perera-2023}, every bounded sequence $\{u_n\}$ in $\Wp{p}$ admits a subsequence such that $f(\cdot,u_n)\to f(\cdot,u)$ in $\Lp{1}$ for some $u\in\Wp{p}$, so that, by \eqref{embedding},
	\[
	\norm{\Psi'(u_n)-\Psi'(u)}_{(\Wp{p})^*}\le k_a\norm{f(\cdot,u_n)-f(\cdot,u)}_1\to0.
	\]
	
	Thus, our aim is to find critical points of $I_\lambda$ for some $\lambda>0$ and our main tools are two three critical points theorems which we recall below. The first one was obtained by Bonanno--Marano \cite[Theorem 3.6]{Bonanno-Marano-2010} and requires the coercivity of the functional $\Phi-\lambda\Psi$.

	\begin{theorem} \label{3pc1}
		Let $X$ be a reflexive real Banach space, \mbox{$\Phi\colon X\rightarrow\R$} be a coercive, continuously G\^{a}teaux differentiable and sequentially weakly lower semicontinuous functional whose G\^{a}teaux derivative admits a continuous inverse on $X^*$, $\Psi\colon X\rightarrow \R$ be a continuously G\^{a}teaux differentiable functional whose G\^{a}teaux derivative is compact, such that
		\begin{enumerate}[label=\textnormal{(b$_\arabic*$)},ref=\textnormal{b$_\arabic*$}] 
		\item \label{b1} $\inf\limits_X \Phi=\Phi(0)=\Psi(0)=0$.
		\end{enumerate}
		Assume that there exist a positive constant $r$ and $\overline{v}\in X$, with $r<\Phi(\overline{v})$, such that
		\begin{enumerate}[label=\textnormal{(b$_\arabic*$)},ref=\textnormal{b$_\arabic*$}] \setcounter{enumi}{1}
			\item \label{b2} $\dfrac{\sup\limits_{u\in \Phi^{-1}((-\infty,r))}\Psi(u)}{r} < \dfrac{\Psi(\overline{v})}{\Phi(\overline{v})}$,
			\item \label{b3} $\Phi-\lambda\Psi$ is coercive for all $\lambda\in \Lambda_{r}$, where
			\[
			\Lambda_{r} = \left( \frac{\Phi(\overline{v})}{\Psi(\overline{v})},\frac{r}{\sup\limits_{u\in \Phi^{-1}((-\infty,r))}\Psi(u)} \right).
			\]
		\end{enumerate}
		Then, for each $\lambda\in \Lambda_{r}$, the functional $\Phi-\lambda\Psi$ admits at least three distinct critical points.
	\end{theorem}
	
	The second three critical points theorem, due to Averna--Bonanno \cite[Theorem 5.1]{Averna-Bonanno-2009}, requires a suitable hypothesis on the functional $\Psi$. We recall it in the form given in Bonanno--Candito \cite[Theorem 3.3, Remarks 3.9, 3.10]{Bonanno-Candito-2008}.
	
	\begin{theorem} \label{3pc2}
		Let $X$ be a reflexive real Banach space, \mbox{$\Phi\colon X\rightarrow \R$} be a convex, coercive and continuously G\^{a}teaux differentiable functional whose derivative admits a continuous inverse on $X^*$, $\Psi\colon X\rightarrow \R$ be a continuously G\^{a}teaux differentiable functional whose derivative is compact, such that
		\begin{enumerate}[label=\textnormal{(e$_\arabic*$)},ref=\textnormal{e$_\arabic*$}]
			\item \label{e1} $\inf\limits_X \Phi=\Phi(0)=\Psi(0)=0;$
			\item \label{e2} for each $\lambda>0$ and for every $u_1,u_2$ which are local minima of the functional $\Phi-\lambda\Psi$, such that $\Psi(u_1)\geq 0$ and $\Psi(u_2)\geq 0$, one has
			\[\inf\limits_{s\in[0,1]}\Psi(su_1+(1-s)u_2)\geq 0.\]
		\end{enumerate}
		Assume that there are two positive constants $r_1,r_2$ and
		$\overline{v}\in X$, with \mbox{$2r_1<\Phi(\overline{v})<\dfrac{r_2}{2}$}, such that
		\begin{enumerate}[label=\textnormal{(e$_\arabic*$)},ref=\textnormal{e$_\arabic*$}] \setcounter{enumi}{2}
			\item \label{e3} $\dfrac{\sup\limits_{u\in \Phi^{-1}((-\infty,r_1))}\Psi(u)}{r_1}<\dfrac{2}{3}\dfrac{\Psi(\overline{v})}{\Phi(\overline{v})}$,
			
			\item \label{e4} $\dfrac{\sup\limits_{u\in \Phi^{-1}((-\infty,r_2))}\Psi(u)}{r_2}<\dfrac{1}{3}\dfrac{\Psi(\overline{v})}{\Phi(\overline{v})}$.
		\end{enumerate}
		Then, for each
		\[
		\lambda\in \Lambda_{r_1,r_2} =
		\left( \frac{3}{2}\frac{\Phi(\overline{v})}{\Psi(\overline{v})} \, , \, \min\left\{\frac{r_1}{\sup\limits_{u\in \Phi^{-1}((-\infty,r_1))}\Psi(u)} , \frac{r_2/2}{\sup\limits_{u\in \Phi^{-1}((-\infty,r_2))} \Psi(u)}\right\} \right)
		\]
		the functional $\Phi-\lambda\Psi$ admits at least three distinct critical points lying in $\Phi^{-1}((-\infty,r_2))$.
	\end{theorem}
	
	In Theorems~\ref{3pc1} and \ref{3pc2}, as well as in the rest of the paper, we adopt the convention $\frac{\tau}{0}=+\infty$ for every $\tau>0$.
	
\section{Main results} \label{Sec:main}
	In this section, we present our main results on the existence of three weak solutions. To this end, for any $R>0$ put
	\begin{equation} \label{K_R}
		K_R= \dfrac{1}{k_a^p} \dfrac{\Gamma\left( 1 + \frac{N}{2} \right)}{ \pi^{\frac{N}{2}}} \left( \dfrac{R^{p-N}}{2^p - 2^{p-N} + R^p \norm{a}_\infty} \right),
	\end{equation}
	and, for any $x_0 \in \RN$, set
	\begin{align*}
		S= B(x_0,R) \setminus B\left(x_0, \frac{R}{2}\right),
	\end{align*}
	where, here and in the sequel, $B(x_0,\rho)$ denotes the open ball centered at $x_0$ with radius $\rho$, and $\ms(\cdot)$ the Lebesgue measure in $\RN$.
	The first result guarantees the existence of three weak solutions of problem \eqref{problem} requiring an algebraic condition on the nonlinear term (hypothesis \eqref{hp:cond1}) together with a $p$-sublinear growth at infinity. 

	\begin{theorem} \label{3sol_1}
		Let \eqref{hp-problem} be satisfied and assume that there exist $x_0 \in \RN$, $R>0$ and two constants $c, d$ such that
		\begin{enumerate}[itemsep=3pt, label=\textnormal{($h_\arabic*$)},ref=\textnormal{$h_\arabic*$}]
			\item\label{hp:cost}  $0<c<d$,
			
			\item\label{hp:segno} $F(x,t) \ge 0$ for a.e. $x \in S$ and for all $t \in [0,d]$,
			
			\item\label{hp:cond1} $ \dfrac{\displaystyle\intr \max\limits_{|t|\le c} F(x,t) \diff x}{c^p}  < K_R \dfrac{\displaystyle\int_{B\left(x_0, \frac{R}{2}\right)} F(x, d) \diff x}{d^p}$,
			
			\item\label{hp:sublin} $\limsup\limits_{|t|\to+\infty} \left( \esssup\limits_{x \in \RN} \dfrac{F(x,t)}{|t|^p} \right) \le 0$.
		\end{enumerate}
		Then, for each $\lambda\in (\lambda_1, \lambda_2)$, where
		\begin{equation} \label{Lambda}
			\lambda_1 = \frac{1}{p k_a^p K_R} \dfrac{d^p}{\displaystyle\int_{B\left(x_0, \frac{R}{2}\right)} F(x, d) \diff x}, \qquad \lambda_2= 	\frac{1}{p k_a^p} \dfrac{c^p}{\displaystyle\intr \max\limits_{|t|\le c} F(x,t) \diff x},
		\end{equation}
		problem \eqref{problem} admits at least three distinct weak solutions.
	\end{theorem}

\begin{proof}
	We apply Theorem~\ref{3pc1} to show that, for every $\lambda \in (\lambda_1, \lambda_2)$, the functional $I_\lambda$ admits at least three critical points, that is, problem \eqref{problem} admits at least three weak solutions. As observed in Section~\ref{Sec:prel}, the functionals $\Phi$ and $\Psi$ satisfy the regularity assumptions required in Theorem~\ref{3pc1}, and hypothesis \eqref{b1} clearly holds true.
	In order to verify \eqref{b2}, set
	\[
	r = \dfrac{1}{p} \left( \dfrac{c}{k_a} \right)^p.
	\]
	For any $u \in \Phi^{-1}((-\infty,r))$, by using \eqref{embedding} we have
	\begin{align*}
		\norm{u}_\infty \le k_a \norm{u} \le c,
	\end{align*}
	so that
	\begin{align*}
		\Phi^{-1}((-\infty,r)) \subseteq \left\{ u \in \Wp{p} \, \colon \, \norm{u}_\infty \le c  \right\}.
	\end{align*}
	Hence, one has
	\begin{align*}
		\sup\limits_{u\in \Phi^{-1}((-\infty,r))} \Psi(u) \le \sup\limits_{\norm{u}_\infty \le c} \intr F(x,u(x)) \diff x \le \intr \max\limits_{|t|\le c} F(x,t) \diff x
	\end{align*}
	and
	\begin{align} \label{dim1-cond1}
		\frac{\sup\limits_{u\in \Phi^{-1}((-\infty,r))} \Psi(u) }{r} \le p k_a^p \frac{\intr \max\limits_{|t|\le c} F(x,t) \diff x}{c^p}.
	\end{align}
	Now, we introduce $\bar{v} \colon \RN \to \R$ defined by
	\begin{equation} \label{barv}
	\bar{v}(x) = \begin{cases}
		0 &\text{if} \; x \in \RN \setminus B(x_0,R),\\
		\dfrac{2d}{R}(R-|x-x_0|) \quad & \text{if} \; x \in S,\\
		d & \text{if} \; x \in B\left(x_0, \frac{R}{2}\right).
	\end{cases}
	\end{equation}
	From Brezis \cite[Remark 4(ii), p. 265]{Brezis-book-2011} it follows that $\bar{v} \in \Wp{p}$. By the definition of $\bar{v}$, we have that
	\begin{align*}
		\Phi(\bar{v}) &= \dfrac{1}{p} \left( \int_S |\nabla \bar{v}|^p \diff x + \int_S a(x) |\bar{v}|^p \diff x + \int_{B\left(x_0, \frac{R}{2}\right)} a(x) d^p \diff x  \right)\\
		&\le \frac{1}{p} \left[ \left(\frac{2d}{R}\right)^p \ms(S) \left(1 +  \norm{a}_\infty \left(\frac{R}{2}\right)^p \right) + d^p \norm{a}_\infty \ms\left( B\left(x_0, \frac{R}{2}\right) \right) \right]\\
		&= \frac{d^p}{p} \, \frac{\pi^{\frac{N}{2}}}{\Gamma\left( 1 + \frac{N}{2} \right)} \, R^N \, \left[ \frac{2^p}{R^p} \, \frac{2^N-1}{2^N} + \norm{a}_\infty \, \frac{2^N-1}{2^N} + \norm{a}_\infty \, \frac{1}{2^N} \right]\\
		&= \frac{d^p}{p} \, \frac{\pi^{\frac{N}{2}}}{\Gamma\left( 1 + \frac{N}{2} \right)} \, \frac{2^p-2^{p-N}+ R^p\norm{a}_\infty}{R^{p-N}},
	\end{align*}
	namely
	\begin{align} \label{dim1-cond2}
		\Phi(\bar{v}) \le \frac{d^p}{p} \, \frac{1}{k_a^p \, K_R}.
	\end{align}
	Moreover, it is easy to verify that $\Phi(\bar{v}) > r$. Indeed, $\norm{\bar{v}}_\infty=d>c$ by \eqref{hp:cost}, so \eqref{embedding} gives
	\begin{align} \label{dim1-cost}
		\Phi(\bar{v}) = \frac{1}{p} \norm{\bar{v}}^p \ge \frac{1}{p} \frac{1}{k_a^p} \norm{\bar{v}}_\infty^p = \frac{1}{p} \frac{1}{k_a^p} d^p > \frac{1}{p} \frac{1}{k_a^p} c^p =r.
	\end{align}
	Furthermore, by \eqref{hp:segno},
	\begin{align} 
		\Psi(\bar{v}) &= \int_S F\left( x,  \dfrac{2d}{R}(R-|x-x_0|)\right) \diff x + \int_{B\left(x_0, \frac{R}{2}\right)} F(x,d) \diff x  \notag \\
		&\ge \int_{B\left(x_0, \frac{R}{2}\right)} F(x,d) \diff x. \label{dim1-cond3}
	\end{align}
	Putting together \eqref{dim1-cond1}, \eqref{hp:cond1}, \eqref{dim1-cond2} and \eqref{dim1-cond3}, we obtain
	\begin{align*}
			\frac{\sup\limits_{u\in \Phi^{-1}((-\infty,r))} \Psi(u) }{r} &\le p k_a^p \frac{\displaystyle\intr \max\limits_{|t|\le c} F(x,t) \diff x}{c^p}\\
			& < p k_a^p K_R \frac{ \displaystyle\int_{B\left(x_0, \frac{R}{2}\right)} F(x,d) \diff x}{d^p}\\
			& \le \frac{\Psi(\bar{v})}{\Phi(\bar{v})},
	\end{align*}
	and assumption \eqref{b2} of Theorem~\ref{3pc1} holds true, with $(\lambda_1, \lambda_2) \subseteq \Lambda_{r}$. Finally, we show that $I_\lambda$ is coercive for every $\lambda>0$, so that \eqref{b3} is verified. Fix $\lambda>0$ and $\eps \in \left(0, \frac{a_-}{p\lambda}\right)$. By \eqref{hp:sublin} there exists $\delta>0$ such that $F(x,t) \le \eps |t|^p$ for a.e. $x \in \RN$ and all $|t|>\delta$. On the other hand, $F(x,t) \le \delta\, m_\delta(x)$ for a.e. $x\in\RN$ and all $|t| \le \delta$, where $m_\delta = \sup_{|t|\le\delta}|f(\cdot,t)|$ belongs to $\Lp{1}$ by \eqref{hp-problem}. Hence, since $a_- \norm{u}_p^p \le \norm{u}^p$, we get
	\[
	I_\lambda(u) \ge \frac{1}{p}\norm{u}^p - \lambda \eps \norm{u}_p^p - \lambda \delta \norm{m_\delta}_1 \ge \left( \frac{1}{p} - \frac{\lambda\eps}{a_-} \right) \norm{u}^p - \lambda \delta \norm{m_\delta}_1
	\]
	for all $u \in \Wp{p}$, and the coercivity of $I_\lambda$ follows. Then, Theorem~\ref{3pc1} ensures that for any $\lambda \in (\lambda_1, \lambda_2)$ the energy functional $I_\lambda$ admits at least three distinct critical points, that is, problem \eqref{problem} admits at least three distinct weak solutions and the proof is complete.
\end{proof}

\begin{remark}
		Note that the left-hand side of \eqref{hp:cond1} is nonnegative, since $F(x,0)=0$. Hence, \eqref{hp:cond1} implies $\int_{B\left(x_0, \frac{R}{2}\right)} F(x, d) \diff x >0$, so that $\lambda_1$ is well defined, and the interval $(\lambda_1, \lambda_2)$ is nonempty. In particular, $\lambda_2 = +\infty$ if $\intr \max_{|t|\le c} F(x,t) \diff x =0$.
\end{remark}

Adding a hypothesis on the sign of the nonlinear term, we are able to provide information on the sign of the solutions. To this end, we recall the following result obtained in Amoroso--Bonanno--Perera \cite[Lemma 2.2]{Amoroso-Bonanno-Perera-2023}.
\begin{lemma} \label{solnonneg}
	Let \eqref{hp-problem} be satisfied and assume that
	\begin{equation} \label{trunc}
		f(x,t)=f(x,0) \qquad \text{for a.e. } x \in \RN \text{ and all } t \le 0.
	\end{equation}
	If $f(x,0)\ge0$ for a.e. $x \in \RN$, then any weak solution of \eqref{problem} is nonnegative.
\end{lemma}

\begin{remark} \label{rem:trunc}
Condition \eqref{trunc} cannot be removed in Lemma~\ref{solnonneg}: for instance, if $f(x,\cdot)$ is odd, then $-u$ is a weak solution of \eqref{problem} whenever $u$ is. However, \eqref{trunc} is not restrictive for our purposes. Indeed, consider the truncated function
\[
\hat{f}(x,t) = \begin{cases} f(x,t) & \text{if } t > 0,\\ f(x,0) & \text{if } t \le 0, \end{cases}
\]
which satisfies \eqref{hp-problem} and \eqref{trunc}, and denote by $(\hat{P}_\lambda)$ problem \eqref{problem} with $\hat{f}$ in place of $f$. Assume that $f(x,0)\ge 0$ for a.e. $x\in\RN$. Then, by Lemma~\ref{solnonneg}, every weak solution $u$ of $(\hat{P}_\lambda)$ is nonnegative, so that $\hat{f}(x,u(x)) = f(x,u(x))$ for a.e. $x\in\RN$ and $u$ is a nonnegative weak solution of \eqref{problem}. Moreover, setting $\hat{F}(x,t)=\int_0^t \hat{f}(x,\xi)\diff\xi$, one has $\hat{F}(x,t)=F(x,t)$ for $t\ge0$ and $\hat{F}(x,t)=f(x,0)\,t\le0$ for $t\le0$, hence, for a.e. $x\in\RN$,
\begin{equation} \label{trunc-max}
	\hat{F}(x,t) \le \max\{F(x,t),0\} \;\; \forall\, t\in\R, \qquad \max_{|t|\le c}\hat{F}(x,t) \le \max_{|t|\le c} F(x,t) \;\; \forall\, c>0.
\end{equation}
\end{remark}

Note that, by \eqref{trunc-max}, if $f$ satisfies \eqref{hp:cost}--\eqref{hp:sublin} and $f(x,0)\ge0$ for a.e. $x \in \RN$, then $\hat{f}$ satisfies \eqref{hp:cost}--\eqref{hp:sublin} as well, with the same $\lambda_1$ and a value of $\lambda_2$ not smaller than the one in \eqref{Lambda}. Hence, applying Theorem~\ref{3sol_1} to $\hat{f}$ and taking Remark~\ref{rem:trunc} into account, we obtain the following consequence.

\begin{corollary} \label{th1:nonneg}
		Let \eqref{hp-problem} be satisfied and assume that there exist $x_0 \in \RN$, $R>0$ and two constants $c, d$ such that \eqref{hp:cost}--\eqref{hp:sublin} hold true. Moreover, suppose that
				\begin{enumerate}[itemsep=3pt, label=\textnormal{($h_\arabic*$)},ref=\textnormal{$h_\arabic*$}]
					\setcounter{enumi}{4}
					\item\label{hp:f0} $f(x,0)\ge0$ for a.e. $x \in \RN$.
				\end{enumerate}
		Then, for every $\lambda\in (\lambda_1, \lambda_2)$, defined in \eqref{Lambda}, problem \eqref{problem} admits at least three distinct nonnegative weak solutions.
\end{corollary}

\begin{remark} \label{rem:nontrivial}
	In Theorem~\ref{3sol_1} and Corollary~\ref{th1:nonneg} (as well as in Theorem~\ref{3sol_2} below), if $f(\cdot,0) \not\equiv 0$, then $u \equiv 0$ is not a solution of \eqref{problem} and, therefore, all three solutions obtained are nontrivial. If $f(\cdot,0) \equiv 0$, then $u\equiv 0$ is a solution of \eqref{problem} and the results ensure the existence of at least two nontrivial solutions.
\end{remark}

Our second existence result is obtained by requiring an appropriate growth of the nonlinear term between the levels $c_1$ and $d$ and between the levels $d$ and $c_2$ (see \eqref{hp2:cond1} and \eqref{hp2:cond2}).

\begin{theorem} \label{3sol_2}
	Let \eqref{hp-problem} be satisfied and assume that there exist $x_0 \in \RN$, $R>0$ and three positive constants $c_1, c_2, d$ such that
	\begin{enumerate}[itemsep=4pt, label=\textnormal{($h'_\arabic*$)},ref=\textnormal{$h'_\arabic*$}]
		\item\label{hp2:costanti} $2^{\frac{1}{p}} c_1 < d < c_2$,
				
		\item\label{hp2:segno} $f(x,t) \ge 0$ for a.e. $x \in \RN$ and for all $t \ge 0$,
		
		\item\label{hp2:cond1} $ \dfrac{\displaystyle\intr \max\limits_{|t|\le c_1} F(x,t) \diff x}{c_1^p}  < \dfrac{2}{3} K_R \dfrac{\displaystyle\int_{B\left(x_0, \frac{R}{2}\right)} F(x, d) \diff x}{d^p}$,
		
		\item\label{hp2:cond2} $ \dfrac{\displaystyle\intr \max\limits_{|t|\le c_2} F(x,t) \diff x}{c_2^p}  < \dfrac{1}{3} K_R \dfrac{\displaystyle\int_{B\left(x_0, \frac{R}{2}\right)} F(x, d) \diff x}{d^p}$.
	\end{enumerate}
	Then, for each $\lambda\in \left(\lambda'_1, \lambda'_2\right)$, where
	\begin{align*}
		\lambda'_1 &= \frac{3}{2}\frac{1}{p k_a^p K_R} \dfrac{d^p}{\displaystyle\int_{B\left(x_0, \frac{R}{2}\right)} F(x, d) \diff x}, \\
		\lambda'_2 &= \min\left\{ \frac{1}{p k_a^p} \dfrac{c_1^p}{\displaystyle\intr \max\limits_{|t|\le c_1} F(x,t) \diff x} \, , \, \frac{1}{2 p k_a^p} \dfrac{c_2^p}{\displaystyle\intr \max\limits_{|t|\le c_2} F(x,t) \diff x}\right\},
	\end{align*}
	problem \eqref{problem} admits at least three distinct nonnegative weak solutions 
	\[
	u_i \in \Wp{p}, \quad i=1,2,3
	\]
	such that $\norm{u_i} < \dfrac{c_2}{k_a}$ and $\norm{u_i}_\infty < c_2$ for every $i=1,2,3$.
\end{theorem}

\begin{proof}
	First, note that $f(x,0)\ge0$ for a.e. $x\in\RN$ by \eqref{hp2:segno}. Hence, by Remark~\ref{rem:trunc}, we may assume that $f$ satisfies \eqref{trunc}: indeed, $\hat{f}$ satisfies \eqref{hp-problem}, \eqref{hp2:costanti} and \eqref{hp2:segno}; by \eqref{trunc-max} it also satisfies \eqref{hp2:cond1} and \eqref{hp2:cond2}, with a value of $\lambda'_2$ not smaller than the one of $f$ and the same $\lambda'_1$; finally, the weak solutions of $(\hat{P}_\lambda)$ are nonnegative weak solutions of \eqref{problem}.

	The proof is based on applying Theorem~\ref{3pc2} to the functionals $\Phi, \Psi$ defined in Section~\ref{Sec:prel}, which, as observed there, satisfy the required regularity assumptions and \eqref{e1}. Also, \eqref{e2} holds true by the definition of $\Psi$ and assumption \eqref{hp2:segno}. Indeed, fix $\lambda>0$ and let $u_1, u_2 \in \Wp{p}$ be local minima of $I_\lambda = \Phi-\lambda\Psi$; then they are critical points of $I_\lambda$, that is, weak solutions of \eqref{problem}, and since \eqref{hp2:segno} holds, by Lemma~\ref{solnonneg} it follows that $u_1 \ge 0$ and $u_2 \ge 0$. Again from \eqref{hp2:segno}, we have that $F(x,t) \ge 0$ for a.e. $x \in \RN$ and all $t \ge 0$, which leads to $\Psi(u_1)\ge0$ and $\Psi(u_2)\ge0$. Moreover,
	\[
	s u_1 + (1-s) u_2 \ge 0 \qquad \text{for any} \; s \in [0,1],
	\]
	which yields
	\[
	F(x,s u_1 + (1-s) u_2) \ge 0 \qquad \text{for a.e.} \; x \in \RN,
	\]
	and
	\[
	\Psi(s u_1 + (1-s) u_2) \ge 0 \qquad \text{for any} \; s \in [0,1].
	\]
	Now, in order to prove \eqref{e3} and \eqref{e4} we proceed as in the proof of Theorem~\ref{3sol_1}. Note that the left-hand side of \eqref{hp2:cond1} is nonnegative, since $F(x,0)=0$. Hence, \eqref{hp2:cond1} implies $\int_{B\left(x_0, \frac{R}{2}\right)} F(x,d) \diff x > 0$, so that $\lambda'_1$ is well defined, and the interval $(\lambda'_1, \lambda'_2)$ is nonempty by \eqref{hp2:cond1}--\eqref{hp2:cond2}. Set $r_1 = \dfrac{1}{p} \left( \dfrac{c_1}{k_a} \right)^p$ and $r_2 = \dfrac{1}{p} \left( \dfrac{c_2}{k_a} \right)^p$. Then, by \eqref{embedding},
	\begin{align*}
		&\norm{u}_\infty \le k_a \norm{u} \le c_1 \qquad \text{for all} \; u \in \Phi^{-1}((-\infty, r_1)),\\
		&\norm{u}_\infty \le k_a \norm{u} \le c_2 \qquad \text{for all} \; u \in \Phi^{-1}((-\infty, r_2)).
	\end{align*}
	Then, taking $\bar{v} \in \Wp{p}$ defined as in \eqref{barv}, from assumptions \eqref{hp2:cond1} and \eqref{hp2:segno} we obtain
	\begin{align*}
		\frac{\sup\limits_{u\in \Phi^{-1}((-\infty,r_1))} \Psi(u) }{r_1} &\le p k_a^p \frac{\displaystyle\intr \max\limits_{|t|\le c_1} F(x,t) \diff x}{c_1^p}\\
		& < \frac{2}{3} p k_a^p K_R \frac{\displaystyle \int_{B\left(x_0, \frac{R}{2}\right)} F(x,d) \diff x}{d^p}\\
		& \le\frac{2}{3} \frac{\Psi(\bar{v})}{\Phi(\bar{v})}.
	\end{align*}
	Similarly, by using assumptions  \eqref{hp2:cond2} and \eqref{hp2:segno} we have
	\begin{align*}
		\frac{\sup\limits_{u\in \Phi^{-1}((-\infty,r_2))} \Psi(u) }{r_2} &\le p k_a^p \frac{\displaystyle\intr \max\limits_{|t|\le c_2} F(x,t) \diff x}{c_2^p}\\
		& < \frac{1}{3} p k_a^p K_R \frac{\displaystyle \int_{B\left(x_0, \frac{R}{2}\right)} F(x,d) \diff x}{d^p}\\
		& \le\frac{1}{3} \frac{\Psi(\bar{v})}{\Phi(\bar{v})}.
	\end{align*}
	Finally, we show that $2r_1 < \Phi(\bar{v}) < \frac{r_2}{2}$. On the one hand, from \eqref{hp2:costanti} and \eqref{dim1-cost} one has
	\[
	2 r_1 = \frac{2}{p} \frac{c_1^p}{k_a^p} < \frac{1}{p}\frac{d^p}{k_a^p} \le \Phi(\bar{v}).
	\]
	On the other hand, by \eqref{hp2:segno}, for a.e. $x \in \RN$ the function $F(x,\cdot)$ is nonnegative and nondecreasing in $[0,+\infty)$. Hence, since $d<c_2$, we get
	\[
	\int_{B\left(x_0, \frac{R}{2}\right)} F(x,d) \diff x \le \intr F(x,c_2) \diff x \le \intr \max\limits_{|t|\le c_2} F(x,t) \diff x,
	\]
	which, together with \eqref{hp2:cond2}, gives
	\[
	\frac{1}{c_2^p} \int_{B\left(x_0, \frac{R}{2}\right)} F(x,d) \diff x < \frac{K_R}{3} \, \frac{1}{d^p} \int_{B\left(x_0, \frac{R}{2}\right)} F(x,d) \diff x.
	\]
	Since $\int_{B\left(x_0, \frac{R}{2}\right)} F(x,d) \diff x > 0$, we get $d^p < \frac{K_R}{3} c_2^p$. Therefore, by \eqref{dim1-cond2},
	\[
	\Phi(\bar{v}) \le \frac{1}{p} \frac{d^p}{k_a^p K_R} < \frac{1}{3p} \frac{c_2^p}{k_a^p} < \frac{r_2}{2}.
	\]
	Hence, all the assumptions of Theorem~\ref{3pc2} are satisfied. Moreover, by the first inequality in each of the two chains above and by \eqref{dim1-cond2}, \eqref{dim1-cond3}, one has $(\lambda'_1, \lambda'_2) \subseteq \Lambda_{r_1,r_2}$. Therefore, by applying Theorem~\ref{3pc2} we get that for any $\lambda \in (\lambda'_1, \lambda'_2)$ problem \eqref{problem} admits at least three distinct weak solutions such that $\norm{u_i} < \dfrac{c_2}{k_a}$ and $\norm{u_i}_\infty < c_2$ for every $i=1,2,3$. Also, the solutions are nonnegative thanks to \eqref{hp2:segno} and Lemma~\ref{solnonneg}.
\end{proof}

In the following, we highlight some consequences in the particular case where the nonlinearity has separated variables, namely
\[
f(x,t)=h(x) g(t) \qquad \text{for all} \; x \in \RN, t \in \R.
\]
We specialize the existence results presented so far to the following problem
	\begin{equation*} \label{problem-SV} \tag{$\tilde{P}_\lambda$}
	\begin{cases}
		-\Delta_p u + a(x) |u|^{p-2} u = \lambda h(x) g(u) \quad \text{in} \quad  \RN,	\\
		u \in \Wp{p}, u \ge 0,
	\end{cases}
\end{equation*}
assuming that
\begin{enumerate}[label=\textnormal{($\tilde{H}$)},ref=\textnormal{$\tilde{H}$}] 
	\item\label{hp-problem-SV} $N< p < \infty$, $a \in \Lp{\infty}$ with $a_- = \essinf\limits_{\R^N} a >0$,  $h \in \Lp{1}$ with $h \ge 0$, $h \not\equiv 0$, and $g \in C(\R)$ with $g \ge 0$ and $g \not\equiv 0$ in $[0,+\infty)$.
\end{enumerate}

Now, setting $G(t)=\int_0^t g(\xi) \diff \xi$ for any $t \in \R$, we have that $F(x,t)=h(x)G(t)$ for all $(x,t) \in \RN \times \R$. Moreover, since $g$ is nonnegative in $[0,+\infty)$, $G$ is nondecreasing and nonnegative in $[0,+\infty)$. Hence, since $\hat{F}(x,t)=h(x)g(0)\,t\le0$ for $t\le0$ (see Remark~\ref{rem:trunc}), one has
\begin{equation} \label{max-SV}
	\max_{|t|\le c}\hat{F}(x,t) = h(x)G(c) \qquad \text{for a.e. } x\in\RN \text{ and every } c>0.
\end{equation}
Since $h \ge 0$ and $h \not\equiv 0$, there exist $x_0 \in \RN$ and $R>0$ such that $\int_{B\left(x_0, \frac{R}{2}\right)} h(x) \diff x >0$; for any such $x_0$ and $R$, put
\begin{align*}
	\tilde{K}_R = K_R \dfrac{\displaystyle\int_{B\left(x_0, \frac{R}{2}\right)} h(x) \diff x}{\norm{h}_1},
\end{align*}
where $K_R$ is given in \eqref{K_R}. Clearly, $0<\tilde{K}_R \le K_R$. Moreover, since $G(c)\ge0$ for every $c>0$, hypotheses \eqref{hp:cond1-SV} and \eqref{hp2:cond1-SV} below imply $G(d)>0$. Applying Theorem~\ref{3sol_1} to $\hat{f}$ (see Remark~\ref{rem:trunc}), we obtain the following result.
	\begin{corollary} \label{th1-SV}
	Let \eqref{hp-problem-SV} be satisfied and assume that there exist $x_0 \in \RN$ and $R>0$ with $\int_{B\left(x_0, \frac{R}{2}\right)} h(x) \diff x >0$, and two constants $c, d$ such that
	\begin{enumerate}[itemsep=3pt, label=\textnormal{($\tilde{h}_\arabic*$)},ref=\textnormal{$\tilde{h}_\arabic*$}]
				\setcounter{enumi}{1}
		\item[\textnormal{(\ref*{hp:cost})}] $0<c<d$,
		
		\item\label{hp:cond1-SV} 
		$ \dfrac{G(c)}{c^p}  < \tilde{K}_R \dfrac{G(d)}{d^p}$,
		
		\item\label{hp:sublin-SV}
		$\lim\limits_{t \to+\infty} \dfrac{G(t)}{t^p}=0$.
	\end{enumerate}
	Then, for each $\lambda\in (\tilde{\lambda}_1, \tilde{\lambda}_2)$, where
	\begin{equation*}
		\tilde{\lambda}_1 = \frac{1}{p k_a^p \norm{h}_1 \tilde{K}_R} \dfrac{d^p}{G(d)}, \qquad
		 \tilde{\lambda}_2= 	\frac{1}{p k_a^p \norm{h}_1} \dfrac{c^p}{G(c)},
	\end{equation*}
	problem \eqref{problem-SV} admits at least three distinct nonnegative weak solutions.
\end{corollary}
\begin{proof}
	It suffices to apply Theorem~\ref{3sol_1} to $\hat{f}$, with $f(x,t)=h(x)g(t)$, and to take Remark~\ref{rem:trunc} into account. Note that $f(x,0)=h(x)g(0)\ge0$ and that $\hat{f}$ satisfies \eqref{hp:cost}--\eqref{hp:cond1}, with $\lambda_1=\tilde{\lambda}_1$ and $\lambda_2=\tilde{\lambda}_2$, since, by \eqref{max-SV}, $\intr \max_{|t|\le c} \hat{F}(x,t) \diff x = \norm{h}_1 G(c)$, and $\int_{B\left(x_0, \frac{R}{2}\right)} \hat{F}(x,d) \diff x = G(d) \int_{B\left(x_0, \frac{R}{2}\right)} h(x) \diff x$. We only point out that, since $h$ may be unbounded, hypothesis \eqref{hp:sublin} is not directly implied by \eqref{hp:sublin-SV}. However, \eqref{hp:sublin} is used in the proof of Theorem~\ref{3sol_1} only to obtain the coercivity of the energy functional, which still holds. Indeed, fix $\lambda>0$ and $\eps \in \left(0, \frac{1}{p\lambda k_a^p \norm{h}_1}\right)$; by \eqref{hp:sublin-SV} there is $\delta>0$ such that $G(t) \le \eps t^p$ for all $t>\delta$. Since $G(t) \le G(\delta)$ for $t \in [0,\delta]$ and $\hat{F}(x,t)=h(x)g(0)\,t \le 0$ for $t \le 0$, we get $\hat{F}(x,t) \le h(x)\left(G(\delta)+\eps|t|^p\right)$ for a.e. $x\in\RN$ and all $t\in\R$. Hence, by \eqref{embedding},
	\[
	\intr \hat{F}(x,u(x)) \diff x \le \norm{h}_1 G(\delta) + \eps \intr h(x) |u(x)|^p \diff x \le \norm{h}_1 G(\delta) + \eps \norm{h}_1 k_a^p \norm{u}^p,
	\]
	so that
	\[
	\Phi(u) - \lambda \intr \hat{F}(x,u(x)) \diff x \ge \left( \frac{1}{p} - \lambda \eps k_a^p \norm{h}_1 \right) \norm{u}^p - \lambda \norm{h}_1 G(\delta)
	\]
	for all $u \in \Wp{p}$.
\end{proof}
Similarly, applying Theorem~\ref{3sol_2} to $\hat{f}$ and taking \eqref{max-SV} and Remark~\ref{rem:trunc} into account, we obtain the next result.
\begin{corollary}  \label{th2-SV}
	Let \eqref{hp-problem-SV} be satisfied and assume that there exist $x_0 \in \RN$ and $R>0$ with $\int_{B\left(x_0, \frac{R}{2}\right)} h(x) \diff x >0$, and three positive constants $c_1, c_2, d$ such that
	\begin{enumerate}[itemsep=4pt, label=\textnormal{($\tilde{h}'_\arabic*$)},ref=\textnormal{$\tilde{h}'_\arabic*$}]
		\setcounter{enumi}{1}
		\item[\textnormal{(\ref*{hp2:costanti})}] $2^{\frac{1}{p}} c_1 < d < c_2$,
		
		\item \label{hp2:cond1-SV}
		$ \dfrac{G(c_1)}{c_1^p}  < \dfrac{2}{3} \tilde{K}_R \dfrac{G(d)}{d^p}$,
		
		\item \label{hp2:cond2-SV}
		$ \dfrac{G(c_2)}{c_2^p}  < \dfrac{1}{3} \tilde{K}_R \dfrac{G(d)}{d^p}$.
	\end{enumerate}
	Then, for each $\lambda\in \left(\tilde{\lambda}'_1, \tilde{\lambda}'_2\right)$, where
	\begin{align*}
		\tilde{\lambda}'_1 &= \frac{3}{2}\frac{1}{p k_a^p \norm{h}_1 \tilde{K}_R} \dfrac{d^p}{G(d)}, \\
		\tilde{\lambda}'_2 &= \min\left\{ \frac{1}{p k_a^p \norm{h}_1} \dfrac{c_1^p}{G(c_1)} \, , \, \frac{1}{2 p k_a^p \norm{h}_1} \dfrac{c_2^p}{G(c_2)}\right\},
	\end{align*}
	problem \eqref{problem-SV} admits at least three distinct nonnegative weak solutions 
	\[
	u_i \in \Wp{p}, \quad i=1,2,3
	\]
	such that $\norm{u_i} < \dfrac{c_2}{k_a}$ and $\norm{u_i}_\infty < c_2$ for every $i=1,2,3$.
\end{corollary}

Finally, we provide an existence result requiring only that $G$ is $p$-superlinear at zero and $p$-sublinear at infinity.

\begin{theorem}  \label{th3-SV}
	Let \eqref{hp-problem-SV} be satisfied and let $x_0 \in \RN$ and $R>0$ be such that
	\[
	\int_{B\left(x_0, \frac{R}{2}\right)} h(x) \diff x >0.
	\]
	Moreover, assume that
	\begin{enumerate}[itemsep=4pt, label=\textnormal{($\tilde{h}'_\arabic*$)},ref=\textnormal{$\tilde{h}'_\arabic*$}]
		\setcounter{enumi}{3}
		\item\label{hp:limit-SV}
		$ \lim\limits_{t \to 0^+}\dfrac{G(t)}{t^p}  = \lim\limits_{t \to + \infty }\dfrac{G(t)}{t^p} =0 $.
	\end{enumerate}
	Then, for each $\lambda > \lambda^*$, where
	\begin{align*}
		\lambda^* = \frac{3}{2}\frac{1}{p k_a^p \norm{h}_1 \tilde{K}_R} \inf_{d>0} \dfrac{d^p}{G(d)},
	\end{align*}
	problem \eqref{problem-SV} admits at least three distinct nonnegative weak solutions, at least two of which are nontrivial.
\end{theorem}
\begin{proof}
	First of all, note that by \eqref{hp-problem-SV} one has $G(d)>0$ for $d$ large enough, so that $\lambda^*<+\infty$ (with the convention $\frac{d^p}{G(d)}=+\infty$ if $G(d)=0$). Now, fix $\lambda > \lambda^*$, so there exists $d>0$ such that
	\[
		\lambda> \frac{3}{2}\frac{1}{p k_a^p \norm{h}_1 \tilde{K}_R} \dfrac{d^p}{G(d)}.
	\]
	From the first part of hypothesis \eqref{hp:limit-SV} it follows that there is $c_1 \in \left(0, \frac{d}{2^{\frac{1}{p}}}\right)$ such that
	\[
		\dfrac{G(c_1)}{c_1^p} < \dfrac{1}{\lambda} \dfrac{1}{p k_a^p \norm{h}_1} < \dfrac{2}{3} \tilde{K}_R \dfrac{G(d)}{d^p},
	\]
	where the last inequality follows from the choice of $d$; namely, assumption \eqref{hp2:cond1-SV} is verified. On the other hand, from the second part of \eqref{hp:limit-SV} we get that there exists $c_2 > d$ such that
	\[
		\dfrac{G(c_2)}{c_2^p} < \dfrac{1}{2} \dfrac{1}{\lambda} \dfrac{1}{p k_a^p \norm{h}_1} < \dfrac{1}{3} \tilde{K}_R \dfrac{G(d)}{d^p},
	\]
	which implies that assumption \eqref{hp2:cond2-SV} holds true. Finally, by the choice of $c_1$ and $c_2$, assumption \eqref{hp2:costanti} is satisfied. Moreover, $\lambda \in (\tilde{\lambda}'_1, \tilde{\lambda}'_2)$: indeed, $\lambda > \tilde{\lambda}'_1$ by the choice of $d$, while the inequalities satisfied by $\frac{G(c_1)}{c_1^p}$ and $\frac{G(c_2)}{c_2^p}$ give $\lambda < \tilde{\lambda}'_2$. Hence, Corollary~\ref{th2-SV} guarantees the existence of at least three distinct nonnegative weak solutions for any $\lambda > \lambda^*$. In addition, the first limit in \eqref{hp:limit-SV} implies $g(0)=0$: indeed, $\frac{G(t)}{t}\to g(0)$ as $t\to0^+$, while $\frac{G(t)}{t}=\frac{G(t)}{t^p}\,t^{p-1}\to0$. Hence, $u\equiv0$ is a solution of \eqref{problem-SV} and, by Remark~\ref{rem:nontrivial}, at least two of the three solutions are nontrivial.
\end{proof}

\begin{remark} \label{rem:g-SV}
	Note that assumption \eqref{hp:limit-SV} is implied by
	\begin{enumerate}[itemsep=4pt, label=\textnormal{($\tilde{h}'_\arabic*$)},ref=\textnormal{$\tilde{h}'_\arabic*$}]
		\setcounter{enumi}{4}
		\item \label{hp:limit-g-SV}
		$ \lim\limits_{t \to 0^+}\dfrac{g(t)}{t^{p-1}}  = \lim\limits_{t \to + \infty }\dfrac{g(t)}{t^{p-1}} =0 $.
	\end{enumerate}
\end{remark}

\begin{remark}
	Examples of functions $g$ satisfying the hypotheses of Theorem~\ref{th3-SV} are the following
	\begin{align*}
		g_1(t) &=  |t|^{p-1} \dfrac{1-e^{-|t|^\beta}}{1+ |t|^\alpha} \qquad \text{for any} \; \alpha, \beta>0,\\
		g_2(t) & = \dfrac{|t|^{p+\varepsilon-1}}{(1+|t|)^\alpha}  \qquad  \text{for any} \; \alpha > \varepsilon >0,\\
		g_3(t) & = |t|^{p+\varepsilon-1} e^{-|t|^\beta}  \qquad \text{for any} \; \varepsilon, \beta >0,
	\end{align*} 
	for all $t \in \R$.
\end{remark}

\begin{remark} \label{rem:BOV}
	Let $N=1$, $a\equiv B$ and $h=\alpha$, as in \cite{Bonanno-O'Regan-Vetro-2017}. Choosing $x_0=0$ and $R=2$, the function $\bar{v}$ in \eqref{barv} coincides with the one used in \cite{Bonanno-O'Regan-Vetro-2017}, and Corollary~\ref{th2-SV} and Theorem~\ref{th3-SV} have the same structure as \cite[Theorem 3.1]{Bonanno-O'Regan-Vetro-2017} and \cite[Corollary 3.4]{Bonanno-O'Regan-Vetro-2017}, respectively. However, the parameter intervals do not coincide: those in \cite{Bonanno-O'Regan-Vetro-2017} are slightly larger, since there $\Phi(\bar{v})$ is computed exactly and the one-dimensional embedding constant $c_B=\big(\tfrac{p-1}{p}\big)^{(p-1)/p}k_a$ is used. On the other hand, our results allow any choice of $x_0$ and $R$, and Theorem~\ref{th3-SV} requires a weaker condition than \cite[Corollary 3.4]{Bonanno-O'Regan-Vetro-2017} (see Remark~\ref{rem:g-SV}).
\end{remark}


\end{document}